\documentclass[11pt,reqno]{amsart}
\usepackage{graphicx}
\usepackage{verbatim}
\usepackage{textcomp}
\usepackage{amssymb}
\usepackage{cite}
\usepackage{hyperref}
\usepackage{amsmath}
\usepackage{latexsym}
\usepackage{amscd}
\usepackage{amsthm}
\usepackage{mathrsfs}
\usepackage{xypic}
\usepackage{bm}
\usepackage{url}
\newtheorem{thm}{Theorem} [section]
\newtheorem{lem}{Lemma} [section]

\newtheorem{cor}{Corollary} [section]

\newtheorem{remark}{Remark}[section]
\numberwithin{equation}{section}
\newcommand{\R}{\mathbb R}

\newcommand{\ml}{\mathfrak{L}}
\newcommand{\n}{\nabla}
\newcommand{\ld}{\lambda}
\newcommand{\ldp}{\lambda^{(p)}}
\newcommand{\dpx}{\Delta_{p,x}}
\newcommand{\dpf}{\Delta_{p,f}}

\newcommand{\dvolx}{e^{-\frac{|x|^2}{2}}\mbox{dvol}}
\newcommand{\dvolf}{e^{-f}\mbox{dvol}}
\newcommand{\intm}{\int_{M^m}}

\def\ba{\begin{aligned}}
\def\ea{\end{aligned}}
\newcommand{\scal}[2]{\langle{#1},{#2}\rangle}

\newcommand{\abs}[1]{\left|#1\right|}

\begin{document}
\title { Inequalities for eigenvalues of the weighted  Hodge Laplacian}
\author{ Daguang Chen$^*$ and Yingying Zhang}
\thanks{$^*$ This work of the first named author was partially supported by NSFC grant No. 11101234.}
\email{dgchen@math.tsinghua.edu.cn}

\keywords{Eigenvalues, Weighted Hodge Laplacian, Universal inequalities, Self-shrinker}
\subjclass[2000]{35P15; 58J50; 58C40; 58A10}

\address{
Department  of Mathematical Sciences,
Tsinghua University,
Beijing 100084, P.R. China.}

\email{yiz308@lehigh.edu}
\address{
Department of Mathematics,
Lehigh University,
Bethlehem, PA USA 18015.}

\maketitle

\begin{abstract} In this paper, we obtain "universal" inequalities for eigenvalues of the weighted Hodge Laplacian on a compact self-shrinker of Euclidean space. These inequalities generalize the Yang-type and Levitin-Parnovski inequalities for eigenvalues of the Laplacian and Laplacian. From the recursion formula of Cheng and Yang \cite{ChengYang07},
 the Yang-type inequality  for eigenvalues of the weighted Hodge Laplacian  are optimal in the sense of the order of eigenvalues.
\end{abstract}

\section{Introduction}
Let $M^m$ be an $m$-dimensional complete Riemannian manifold
and  $\Omega$ be a bounded domain in $M^m$.
The Dirichlet  eigenvalue problem of Laplacian is given by
\begin{equation}\label{DL}
\left \{ \aligned &\Delta u=-\ld u,\qquad \text{in $\Omega$} \\
&u=0, \qquad\qquad \text{on $\partial \Omega$}. \endaligned
\right.
\end{equation}
 It is well known that the spectrum of this problem is real
 and discrete:
  $$
 0<\lambda_1< \lambda_2\leq \lambda_3\leq \cdots\nearrow\infty,
$$
where each $\lambda_i$ has finite multiplicity which is repeated
according to its multiplicity.

The main
developments were obtained by ~Payne, P\'olya and ~Weinberger \cite{PPW56}, ~Hile and ~Protter \cite{HP} and Yang \cite{Yang}.  In 1956,
~Payne, P\'olya and ~Weinberger \cite{PPW56} proved that
\begin{equation}\label{PPW}
\lambda_{k+1}-\lambda_k
 \leq \frac4{mk}\sum_{i=1}^k\lambda_i.
\end{equation}
 In 1980,  ~Hile and ~Protter \cite{HP} improved (\ref{PPW}) to
\begin{equation}\label{HP}
 \sum_{i=1}^k\frac{\lambda_i}{\lambda_{k+1}-
 \lambda_i}\geq \frac{mk}4.
\end{equation}
In 1991, Yang (see \cite{Yang} and more recently \cite{ChengYang05})  obtained a very sharp inequality
\begin{equation}\label{Y1}
\sum_{i=1}^k(\lambda_{k+1}-\lambda_i)(\lambda_{k+1}-(1+\frac4m)\lambda_i)\leq 0.
\end{equation}

There has been much work dedicated to extending and strengthening
the classical inequalities of Payne-P\'{o}lya-Weinberger,
Hile-Protter and Yang. When $M^m$ is an $m$-dimensional compact
manifold, there are similar results about the eigenvalue
estimates for the ~Laplacian (see, e.g.\cite{LiP,ChengYang05,LeungPF2,ChenCheng08,SHI,YangYau}).
For the compact Riemannian manifolds  isometrically immersed in an Euclidean
space or a sphere, J. M. Lee \cite{Lee} proved  Hile-Protter type bounds
for eigenvalues for Hodge Laplacian on $p$-forms.
In 2002, B. Colbois \cite{Col03} derived a Payne-P\'{o}lya-Weinberger type inequality for
the rough Laplacian. In \cite{IM09}, S.~Ilias and O.~Makhoul obtained inequalities  for the eigenvalues of the Hodge Laplacian.

In 1991, N. Anghel \cite{Anghel} obtained the analogous estimate of (1.2) for
the Dirac operator. In 2009, the Yang-type inequality (\ref{Y1}) was extended to the eigenvalues of Dirac operator by the first  author in \cite{Chen09}.

In \cite{Harl88},  Harrell gave an abstract algebraic argument involving operators, their commutators and traces, which generalize the original PPW arguments. These algebraic ideas were developed in
different contexts to produce many new universal eigenvalues inequalities (see \cite{AshbHer1,HarlMich94,Harl07,HarlStub97,HarlStub10,LP02}).

In present paper, making use of  a theorem  of Ashbaugh and Hermi \cite{AshbHer1}, we obtain the Yang-type inequality for higher order eigenvalues of the weighted Hodge Laplacian  for submanifolds in Euclidean space.

\begin{thm} \label{MainThm1} Let $x:(M^m,g)\longrightarrow (\R^n,{\rm can})$ be a compact self-shrinker,  $\dpx=\Delta_H+\frac12\mathcal{L}_{\n |x|^2}$ (see below  (\ref{CM-operator2})) be  the weighted Hodge Laplacian acting on $p$-forms over $M^m$. Assume that  $\Big\{\ldp_i\Big\}_{i=1}^\infty$ are the eigenvalues of $\dpx$ and $\{\varphi_i\}_{i=1}^\infty$ is a corresponding orthonormal basis of $p$-eigenforms.  We have, for any $p \in \left\{0,1,\dots,m\right\}$,
\begin{equation}\label{HL-Eig1}
\begin{aligned}
m\sum_{i=1}^k\left(\ldp_{k+1}-\ldp_i\right)^2
\leq&\sum_{i=1}^k\left(\ldp_{k+1}-\ldp_i\right)\left(4\ldp_i+2m\right.\\
&-\intm|x|^2|\varphi_i|^2\dvolx\\
&-4\intm\langle\mathfrak{Ric}\varphi_i,\varphi_i\rangle\dvolx\\
&+\left. 4\intm\langle\n(\n\frac{|x|^2}{2})\varphi_i,\varphi_i\rangle\dvolx\right)\\
\end{aligned}
\end{equation}
where $\{e_i\}_{i=1}^{m}$ is a local orthonormal basis of $TM^m$ with respect to the induced
metric $g$ and $\mathfrak{Ric}=-\omega^i\wedge \imath(e_j)R(e_i,e_j)$ (see also \ref{Ric-Cur}) is the curvature operator acting on $p$-forms.
\end{thm}
\begin{remark}
  When $p=0$, i.e., $\ld_i:=\ld^{(0)}_i$ are the eigenvalues of the operator $\ml:=\Delta_{0,x}=\Delta+\langle x,\cdot\rangle$ acting on scalar functions, we have
  \begin{equation}\label{ChengPeng}
  \begin{aligned}
   m\sum_{i=1}^k\left(\ld_{k+1}-\ld_i\right)^2\leq &\sum_{i=1}^k\left(\ld_{k+1}-\ld_i\right)\left(4\ld_i+2m-\intm|x|^2|\varphi_i|^2\dvolx\right)\\
   \leq& \sum_{i=1}^k\left(\ld_{k+1}-\ld_i\right)\left(4\ld_i+2m-\min_{M^n}|x|^2\right),
  \end{aligned}
  \end{equation}
  which is  Theorem 1.1 in \cite{ChengPeng}. Therefore, Theorem \ref{MainThm1} generalizes  eigenvalue estimates from the operator $\ml$  to  the weighted Hodge Laplacian $\dpx$.
\end{remark}

\begin{remark}
   If $\abs{x}=c,(c>0)$, the manifold $M^m$ is a submanifold of sphere $\mathbb S^{n-1}(\frac{1}{c})$ in Euclidean space $\R^n$. Furthermore, the weighted Hodge Laplacian $\dpx$ is reduced to the ordinary one.
\end{remark}

For a compact self-shrinker (see (\ref{shrinker})) in Euclidean space, we have
\begin{cor}\label{SH-inq1} Let $x:(M^m,g)\longrightarrow (\R^n,{\rm can})$ be a compact self-shrinker, $H, h$ be the second fundamental form and the mean curvature of the immersion $x$, respectively.     We have, $p \in \left\{1,\dots,m\right\}$,
\begin{equation}\label{HL-Eig2}
 \ba &\sum_{i=1}^k(\ldp_{k+1}-\ldp_i)^2\\
\leq&\frac 4m\sum_{i=1}^k\left(\ldp_{k+1}-\ldp_i\right)\left[\ldp_i+\frac m2+1\right.\\
&\left.+\intm\left(p|H||h|-\Phi(H,h)-\frac14|x|^2\right)|\varphi_i|^2\dvolx\right]\\
\leq &\frac 4m\sum_{i=1}^k\left(\ldp_{k+1}-\ldp_i\right)\left[\ldp_i+\frac m2+1\right.\\
&\left. +\max_{M^m}\left(p|H||h|-\Phi(H,h)-\frac14|x|^2\right)\right],
\ea\end{equation}
where $ \Phi(H,h)$ is a function depending on the second fundamental form $h$ and the mean curvature $H$  defined in (\ref{Cur-est1}).
\end{cor}

From Theorem \ref{MainThm1}, we can obtain the spectral gaps of the consecutive eigenvalues of the weighted Hodge Laplacian $\dpx$.
\begin{cor}Under the same assumption in Corollary \ref{SH-inq1}, we have
\begin{equation*}
\ba
\ldp_{k+1}-\ldp_k\leq &2\left[\left(\frac{2}{m}\frac1k\sum_{i=1}^k\ldp_i+\frac2m+\frac2m\max_{M^m}\left(p|H||h|
-\Phi(H,h)-\frac14|x|^2\right)\right)^2     \right.\\
&-\left.\left(1+\frac4m\right)\frac1k\sum_{j=1}^k\left(\ldp_j
-\frac1k\sum_{i=1}^k\ldp_i\right)^{2}\right]^{\frac12}
\ea
\end{equation*}
\end{cor}


For the lower order eigenvalues of (\ref{DL}),  in 1956, ~Payne, P\'olya and ~Weinberger \cite{PPW56} proved that for
$\Omega \subset {\Bbb R}^2$,
\begin{equation*}
   \lambda_2+\lambda_3\leq 6\lambda_1,
\end{equation*}
which was extended to domains $\Omega \subset {\Bbb R}^m$ in \cite{Thom}(or see Section 3.2 of \cite{Ashb2})
\begin{equation*}
\sum_{i=1}^m(\ld_{i+1}-\ld_1)\leq 4\lambda_1.
\end{equation*}
There are also a variety of
extensions of results of this type, for examples, see \cite{Brands,ChenCheng08,Chen09,CZ11,ChengYang05, Ashb2,SunChengYang}. Recently, S.~Ilias and O.~Makhoul \cite{IM12} obtained the universal inequality for eigenvalues of the Hodge Laplacian.

In the second part of this paper, by using an algebraic identity deduced by Levitin and Parnovski \cite{LP02}, we can obtain
\begin{thm} \label{MainThm2} Let $x:(M^m,g)\longrightarrow (\R^n,{\rm can})$ be a compact self-shrinker and $\dpx$ be the weighted Hodge Laplacian defined acting on $p$-forms over $M^m$. Assume that  $\Big\{\ldp_i\Big\}_{i=1}^\infty$ are the eigenvalues of $\dpx$ and $\{\varphi_i\}_{i=1}^\infty$ is a corresponding orthonormal basis of $p$-eigenforms. We have, for any
$p \in \left\{0,1,\dots,m\right\}$,
\begin{equation}\label{HL-Eig3}
\begin{aligned}
\sum_{l=1}^m\left(\ldp_{i+l}-\ldp_i\right)
\leq&4\ldp_i+2m-\intm|x|^2|\varphi_i|^2\dvolx\\
&-4\intm\langle\mathfrak{Ric}\varphi_i,\varphi_i\rangle\dvolx\\
  &+4\intm\langle\n(\n\frac{|x|^2}{2})\varphi_i,\varphi_i\rangle\dvolx.
\end{aligned}
\end{equation}
\end{thm}
\begin{remark}
  When $p=0$, i.e., $\ld_i=\ld^{(0)}_i$ is the eigenvalues of the operator $\ml=\Delta_{0,x}=\Delta+\langle x,\cdot\rangle$ acting on scalar functions, we have
  \begin{equation}\label{LL-Eig3}
    \begin{aligned}
    \sum_{l=1}^m(\ld_{i+l}-\ld_i)
\leq&4\ld_i+2m-\intm|x|^2|\varphi_i|^2\dvolx\\
\leq&4\ld_i+2m-\min_{M^m}|x|^2.
    \end{aligned}
  \end{equation}
Since $i$ is arbitrary, (\ref{LL-Eig3}) is more general than Proposition 4.1 in \cite{ChengPeng}.
\end{remark}

\begin{cor}\label{Lower-inq}
 Let $x:(M^m,g)\longrightarrow (\R^n,{\rm can})$ be  a self-shrinker, $H, h$ be the second fundamental form and the mean curvature of the immersion $x$, respectively. Assume that  $\Big\{\ldp_i\Big\}_{i=1}^\infty$ are the eigenvalues of $\dpx$ and $\{\varphi_i\}_{i=1}^\infty$ is a corresponding orthonormal basis of $p$-eigenforms. We obtain, for $p \in \left\{1,\dots,m\right\}$,
\begin{equation}\label{HL-Eig4}
\begin{aligned}
\sum_{l=1}^m\left(\ldp_{i+l}-\ldp_i\right)
\leq&4\ldp_i+2m+4\\
&+4\intm\left(p|H||h|-\Phi(H,h)-\frac14|x|^2\right)|\varphi_i|^2\dvolx\\
\leq &4\ldp_i+2m+4+\max_{M^m}\left(p|H||h|-\Phi(H,h)-\frac14|x|^2\right).
\end{aligned}
\end{equation}
\end{cor}

Furthermore, from the recursion formula of Cheng and Yang \cite{ChengYang07}, we can obtain an upper bound for eigenvalue
$\ldp_k$:
\begin{cor} Let $M^m$ be an $m$-dimensional compact
self-shrinker  in $\mathbb{R}^{n}$. Then,
eigenvalues of the weighted Hodge Laplacian $\dpx$ \ref{CM-operator2}
satisfy, for any $k\geq 1$,
\begin{equation*}
\mu_{k+1}\leq  C_{0}(m) k^{\frac 2m}\mu_1
\end{equation*}
where  $C_{0}(m)\leq 1+\frac{4}{m}$ is a constant and $\mu_i=\ldp_{i}+\frac m2+1+\max_{M^m}\Big(p|H||h|-\Phi(H,h)-\frac14|x|^2\Big)$.
\end{cor}

This paper is organized as follows: In Section 2, we present some formulas for submanifolds in Euclidean space, the definitions of the weighted Hodge Laplacian. In Section 3, in order to prove main theorems, we derive several lemmas for differential forms. In Section 4 and Section 5, we give the proofs of Theorem \ref{MainThm1} and \ref{MainThm2}.

\subsection*{Acknowledgments}
The authors wish to express their gratitude to Professors Huaidong Cao and
Xiaofeng Sun for their suggestions and useful discussions. This work of the first named author was done while the author visited Department of Mathematics, Lehigh University, USA. He also would like to thank the institute for its hospitality.

\section{Preliminaries}
\subsection{Submanifold in Euclidean space and self-shrinker}
Let $x:M^m\to \R^{n}$ be an  $m$-dimensional submanifold of  $n$-dimensional Euclidean space $\R^{n}$. Let $\{e_1,\cdots,e_{m}\}$
be a local orthonormal basis of $TM^m$ with respect to the induced
metric, and $\{\omega^1,\cdots, \omega^m\}$ be their dual 1-forms. Let
$\{e_{m+1},\cdots, e_{n}\}$ be the local orthonormal unit normal
vector fields. In this paper we make the following conventions on the range of
indices:
\begin{equation*}
1\leq i,j,k\leq m; \qquad
m+1\leq \alpha,\beta,\gamma\leq n.
\end{equation*}
Then we have the following structure equations (see \cite{Chen09,ChengPeng})
\begin{equation}\label{Formula-Sub}
\begin{aligned}
&dx=\omega^ie_i,\qquad \omega^\alpha=0,\\
&de_i=\omega^j_ie_j+\omega^\alpha_ie_\alpha,
\qquad \omega^\alpha_i=h^\alpha_{ij}\omega^j,\\
&de_\alpha=\omega^j_\alpha e_j+\omega^\beta_\alpha e_\beta,
\end{aligned}
\end{equation}
where  $h^\alpha_{ij}$ denote the the components of the second
fundamental form  of $M^m$. We denote by $|h|^2=\sum\limits_{\alpha,i,j}(h^\alpha_{ij})^2,$  the norm square of the second fundamental form, $H=\sum\limits_{\alpha}H^\alpha e_\alpha=\sum\limits_{\alpha}
(\sum\limits_i h^\alpha_{ii})e_\alpha$  the mean curvature vector
field over $M^m$.

One can deduce that, ~pointwise on $M^m$,
\begin{equation}\label{xnorm}
 \sum_{A=1}^n|\n x^A|^2=m,
\end{equation}
and
\begin{equation}\label{ssf1}
\frac12|x|^2_{,ij}=\frac12(\sum_{A=1}^n(x^A)^2)_{,ij}=\langle h^\alpha_{ij}e_\alpha,x\rangle+\delta_{ij}.
\end{equation}

The submanifold $M^m$ is called a self-shrinker \cite{CM12} if it satisfies the
quasilinear elliptic system:
\begin{equation}\label{shrinker}
  H=-x^\perp,
\end{equation}
where $H$ denotes the mean curvature vector field of the immersion and $\perp$ is the
projection onto the normal bundle of $M^m$.

\subsection{Differential forms and the weighted Hodge Laplacian}
Let $(M^m,g)$ be an $m$-dimensional compact Riemannian manifold. For any two $p$-forms $\varphi$ and $\psi$, we let $\varphi_{i_1\cdots i_p}=\varphi(e_{i_1},\cdots,e_{i_p})$ and $\psi_{i_1\cdots i_p}=\psi(e_{i_1},\cdots,e_{i_p})$ denote the components of $\varphi$ and $\psi$, with respect to a local orthonormal frame $\{e_i\}_{i=1}^m$. Their pointwise inner product with respect to Riemannian metric $g$ is given by
\begin{equation*}
\begin{aligned}
\langle \varphi,\psi \rangle=&\displaystyle{\sum_{1\leq i_1<\cdots<i_p\leq m}}\varphi_{{i_1}\cdots{i_p}}\; \psi_{{i_1}\cdots{i_p}}\\
=&\frac{1}{p!} \sum_{1 \le i_{1},\dots,i_{p}\le m}\varphi_{{i_1}\cdots{i_p}}\; \psi_{{i_1}\cdots{i_p}}.
\end{aligned}
\end{equation*}
We denote by $\Delta_{p}$ the Hodge Laplacian acting on $p$-forms
\begin{equation}\label{HodgeLap}
\Delta_{p} =(d \, \delta + \delta  d),
\end{equation}
where $d$ is the exterior derivative acting on $p$-forms and $\delta$ is the adjoint of $d$ with respect to Riemannian measure $dvol$.

In \cite{Bueler,Peterson2012}, the operator (\ref{HodgeLap}) is generalized to the weighted Hodge Laplacian acting on differential forms. Let $f \in C^\infty(M^m,\mathbb R)$ be a smooth function defined on $M^m$. When the Riemannian measure is changed from
being $\text{dvol}$ to $e^{-f}{dvol}$, it is natural to define the weighted Hodge Laplacian by
\begin{equation}\label{CM-operator}
\Delta_{p,f}=d\delta'+\delta'd
\end{equation}
where $\delta'=e^{f}\delta e^{-f}$, which is the adjoint operator of the exterior derivative $d$ with respect to Riemannian measure $e^{-f}{dvol}$.

For the weighted Hodge Laplacian, we have the following Bochner-Weitzenb\"{o}ck type formula \cite{Peterson2012}
\begin{equation}\label{WB}
  \ba
 \Delta_{p,f}=&\Delta_p+\mathcal{L}_{\n f}\\
  =&\n^*\n-\omega^i\wedge \imath(e_j)R(e_i,e_j)+\mathcal{L}_{\n f}\\
  =&\n^*_f\n-\omega^i\wedge \imath(e_j)R(e_i,e_j)-\n(\n f)\\
  =&\n^*_f\n+\mathfrak{Ric}-\n(\n f)
  \ea
\end{equation}
where $\mathcal{L}$ is the Lie derivative, $\imath(X)$ for $X\in \Gamma(TM^m)$ is inner product acting on forms, $\n X$ acting on from $\varphi$ is given by
$$\n X  \varphi={X^j}_{,l}\omega^l\wedge \imath(e_j)\varphi$$
and
\begin{equation}\label{Ric-Cur}
\mathfrak{Ric}=-\omega^i\wedge \imath(e_j)R(e_i,e_j).
\end{equation}

With respect to the measure $e^{-f}{dvol}$,  the spectrum of  $\Delta_{p,f}$ consists of a nondecreasing, unbounded sequence of eigenvalues with finite multiplicities
\begin{equation*}
{\rm Spec}(\dpf)=\{0 \le \ldp_1 \le \ldp_2\le \ldp_3 \le \cdots \le \ldp_k \le \cdots \}.
\end{equation*}

Let $x=(x^1,\cdots,x^n):M^m\to \R^{n}$ be an  $m$-dimensional  submanifold of $\R^{n}$.  In this article, we will consider the operator (\ref{CM-operator}) over $M^m$, for $f=e^{\frac12|x|^2}$,
\begin{equation}\label{CM-operator2}
\Delta_{p,x}=d\delta'+\delta'd.
\end{equation}
For $\Delta_{p,x}$ acting on the scalar functions, the operator (\ref{CM-operator2}) \cite{CM12}
\footnote{The Laplacian operator is different in \cite{CM12} with a minus sign.}
is given by
\begin{equation}\label{CM-operator1}
  \ml=\Delta+\langle x,\cdot\rangle=-e^{\frac{|x|^2}{2}}\mbox{div}\left(e^{\frac{-|x|^2}{2}}d \right)=
  \delta'd=\Delta_{0,x}.
\end{equation}
where $\Delta$ is the positive operator.  If $M^m$ is a self-shrinker, we have
\begin{equation}\label{Lap-position}
  \ml x^A=x^A,\qquad A=1,\cdots,n.
\end{equation}

\section{Some lemmas}
In order to prove our main theorems, we will derive some lemmas in this section.

By the direct calculations, we have
\begin{lem}For $f, u\in C^\infty(M,\R)$ and $\varphi\in\bigwedge^{p}(T^*M^m)$, we have
\begin{equation}\label{L1}
  \mathcal{L}_{\n f}(u\varphi)=g(\n f,\n u)\varphi+u\mathcal{L}_{\n f}\varphi.
\end{equation}
\begin{equation}\label{L2}
  [\dpf,u]\varphi= [\Delta_p,u]\varphi+[\mathcal{L}_{\n f},u]\varphi
\end{equation}
\begin{equation}\label{L3}
  \delta_f(u\varphi)=-\imath(\n u)\varphi+u\delta_f\varphi
\end{equation}
where $[\dpf,u]\varphi=\dpf(u\varphi)-u\dpf\varphi$.
\end{lem}

\begin{lem}\label{hess-form} Assuming that $T_{ij}$ is a symmetric 2-tensor, we have, for any $p$-form $\varphi$,
  \begin{equation}\label{n1}
    \ba T_{ij}\omega^i\wedge\imath(e_j)\varphi=&\frac1{p!}\sum_{i_1,\cdots,i_p}(T\varphi)_{i_1\cdots  i_p}\omega^{i_1}\wedge  \cdots\wedge\omega^{i_p}\\
    =&\frac1{(p-1)!}\sum_{i_1,\cdots,i_p}\left(\sum_{j}T_{j i_1}\varphi_{ji_2\cdots  i_p}\right)\omega^{i_1}\wedge \cdots\wedge\omega^{i_p},
  \ea\end{equation}
  and
  \begin{equation}\label{n2}
  \left\langle \sum_{i,j=1}^m T_{ij}\omega^i\wedge\imath(e_j)\varphi,\varphi\right\rangle= \frac{1}{(p-1)!}\sum_{j,i_1,\cdots ,i_p}T_{ji_1}\varphi_{ji_2\cdots i_p}\varphi_{i_1\cdots i_p}
  \end{equation}
where $(T\varphi)_{i_1\cdots i_p}=\sum\limits_{j=1}^m\sum\limits_{k=1}^pT_{j i_k}\varphi_{i_1\cdots j\cdots i_p}$.
\end{lem}
\begin{proof}Assuming that $\varphi=\frac1{p!}\sum\limits_{i_1,\cdots,i_p}\varphi_{i_1\cdots i_p}\omega^{i_1}\wedge \cdots\wedge\omega^{i_p}$, then we get
 \begin{equation*}
   \begin{aligned}
    &T_{ij}\omega^i\wedge\imath(e_j)\varphi\\
    =&\frac1{p!}\sum_{i,j=1} T_{ij}\omega^i\wedge\imath(e_j) \left(\sum_{i_1,\cdots,i_p}\varphi_{i_1\cdots i_p}\omega^{i_1}\wedge \cdots\wedge\omega^{i_p}\right)\\
    =&\frac1{p!}\sum_{i,j,i_1,\cdots,i_p}\sum_{k=1}^{p}(-1)^{k-1}T_{ij}\delta_j^{i_k}\varphi_{i_1\cdots  i_p}\omega^i\wedge\omega^{i_1}\wedge \cdots \wedge\widehat{\omega^{i_k}}\wedge \cdots\wedge\omega^{i_p}\\
    =&\frac1{p!}\sum_{i,i_1,\cdots,\hat{i_k},\cdots i_p}\sum_{k=1}^{p}(-1)^{k-1}T_{ij}\varphi_{i_1\cdots j\cdots  i_p}\omega^i\wedge \omega^{i_1}\wedge \cdots \wedge\widehat{\omega^{i_k}}\wedge \cdots\wedge\omega^{i_p}\\
   =&\frac1{p!}\sum_{i_1,\cdots,i_p}(T\varphi)_{i_1\cdots  i_p}\omega^{i_1}\wedge  \cdots\wedge\omega^{i_p}\\
    =&\frac1{(p-1)!}\sum_{i_1,\cdots,i_p}T_{j i_1}\varphi_{ji_2\cdots  i_p}\omega^{i_1}\wedge \cdots\wedge\omega^{i_p}.
   \end{aligned}
 \end{equation*}
Therefore, we obtain
\begin{equation*}
  \begin{aligned}
  \langle \sum_{i,j=1}^m T_{ij}\omega^i\wedge\imath(e_j)\varphi,\varphi\rangle
  =&\frac{1}{p!}\sum_{i_1,\cdots ,i_p}(T\varphi)_{i_1\cdots i_p}\varphi_{i_1\cdots i_p}\\
  =&\frac{1}{(p-1)!}\sum_{j,i_1,\cdots ,i_p}T_{ji_1}\varphi_{ji_2\cdots i_p}\varphi_{i_1\cdots i_p}.
  \end{aligned}
\end{equation*}
\end{proof}
\begin{lem}\label{norm}Under the same assumptions in Lemma \ref{hess-form}, then we have
  \begin{equation}\label{n3}
    |\sum_{i,j,i_2,\cdots,i_p} T_{ij}\varphi_{i i_2\cdots i_p}\varphi_{ji_2\cdots i_p}|\leq p!|T|\varphi|^2
  \end{equation}
and
\begin{equation}\label{n4}
  \left\langle \sum_{i,j=1}^m T_{ij}\omega^i\wedge\imath(e_j)\varphi,\varphi\right\rangle\leq p|T|\varphi|^2
\end{equation}
where $|T|=\big(\sum\limits_{i,j} T_{ij}^2\big)^{\frac12}$.
\end{lem}
\begin{proof}
  \begin{equation*}
  \ba
   \left|\sum_{i,j,i_2,\cdots,i_p} T_{ij}\varphi_{i i_2\cdots i_p}\varphi_{ji_2\cdots i_p}\right|=&
   \left|\sum_{i_2,\cdots,i_p}\sum_{j}(\sum_i T_{ij}\varphi_{i i_2\cdots i_p})(\varphi_{ji_2\cdots i_p})\right|\\
   \leq &\sum_{i_2,\cdots,i_p} \left(\sum_{j}(\sum_i T_{ij}\varphi_{i i_2\cdots i_p})^2 \sum_k\varphi_{ki_2\cdots i_p}^2\right)^{\frac12} \\                                                          \leq &\sum_{i_2,\cdots,i_p} \left(\sum_{j}\sum_i T_{ij}^2\sum_l\varphi_{l i_2\cdots i_p}^2 \sum_k\varphi_{ki_2\cdots i_p}^2\right)^{\frac12}\\                                                     =&\left(\sum_{i,j} T_{ij}^2\right)^{\frac12}\sum_{k,i_2,\cdots,i_p}\varphi_{k i_2\cdots i_p}^2\\
   =&|T|\sum_{i,i_2,\cdots,i_p}\varphi_{i i_2\cdots i_p}^2\\
   =&p!|T|\varphi|^2.
  \ea\end{equation*}
\end{proof}
\begin{lem}Assume that $x:M^m\longrightarrow \R^n$ is a compact self-shrinker, $H, h$ are the second fundamental form and the mean curvature of the immersion $x$, respectively. We have, for any $p$-form $\varphi$, $p\in \{1,\dots, m\}$
 \begin{equation}\label{n5}
   \ba&\intm\langle\n(\n\frac{|x|^2}{2}))\varphi,\varphi\rangle\dvolx\\
   \leq&\intm\abs{\varphi}^2\dvolx+p\intm |H||h||\varphi|^2\dvolx\\
   \leq&\intm\abs{\varphi}^2\dvolx+p\max_{M^m}|H||h|.
 \ea\end{equation}
\end{lem}
\begin{proof}From (\ref{ssf1}), and taking  $T_{ij}=\langle h^\alpha_{ij}e_\alpha,x\rangle$ in (\ref{n4}), we obtain
  \begin{equation*}
\ba
&\intm\langle\n(\n\frac{|x|^2}{2}))\varphi,\varphi\rangle\dvolx\\
=&\intm\langle \langle h^\alpha_{ij}e_\alpha, x\rangle \omega^i\wedge\imath(e_j)\varphi,\varphi\rangle\dvolx+\intm\abs{\varphi}^2\dvolx\\
\leq&\intm\abs{\varphi}^2\dvolx+p\intm \Big(\sum_{i,j}\langle h^\alpha_{ij}e_\alpha, x\rangle^2\Big)^{\frac12}|\varphi|^2\dvolx\\
=&\intm\abs{\varphi}^2\dvolx+p\intm \Big(\sum_{i,j}\big(H^\alpha h^\alpha_{ij}\big)^2\Big)^{\frac12}|\varphi|^2\dvolx\\
\leq&\intm\abs{\varphi}^2\dvolx+p\intm |H||h||\varphi|^2\dvolx\\
\leq&\intm\abs{\varphi}^2\dvolx+p\max_{M^m}|H||h|.
\ea
\end{equation*}
\end{proof}

Combining  Proposition 4.1 in \cite{Grosj} and Theorem 1.1 in \cite{LeungPF1}, we obtain the estimate of $\mathfrak{Ric}$ (\ref{Ric-Cur}) acting on $p$-forms.
 (c.f. Theorem 3.2 of \cite{IM09})
\begin{lem}
\begin{equation}\label{Cur-est}
\langle\mathfrak{Ric}(\varphi),\varphi\rangle \geq\Phi(h,H)|\varphi|^2,\qquad \varphi\in \textstyle{\bigwedge^p}(T^*M^m),
\end{equation}
where
\begin{equation}\label{Cur-est1}
\ba
\Phi(h,H)=&\bigg\{-p^2\bigg[\Big(\frac{m-5}{4}\Big)|H|^2+|h|^2-\frac{1}{4m^2}\Big(\sqrt{m-1}(m-2)|H|\\
&-2\sqrt{m|h|^2-|H|^2}\,\Big)^2\bigg]-\frac{1}{2}\sqrt{p}(p-1)\Big(|H|^2+|h|^2\Big)\bigg\}.
\ea\end{equation}
\end{lem}

\section{Inequalities for eigenvalues}

In order to obtain the extrinsic bounds of higher order eigenvalues
of the weighted Hodge Laplacian, we firstly introduce the abstract formula
derived by Ashbaugh and Hermi \cite{AshbHer1}.

Let $\mathfrak{H}$ be a complex Hilbert space with inner product $(,)$,
$\mathcal{A}:\mathcal{D}\subset\mathfrak{H} \longrightarrow\mathfrak{H}$ a self-adjoint operator
defined on a dense domain $\mathcal {D}$ that is bounded below and
has a discrete spectrum
$$
\ld_1\leq\ld_2\leq \ld_3\leq \cdots.
$$
Let $\{\mathcal{B}_k:\mathcal{A}(\mathcal{D})\longrightarrow\mathfrak{H}\}_{k=1}^N$ be a
collection of symmetric operators leaving $\mathcal{D}$ invariant
and $\{\varphi_i,\ld_i\}_{i=1}^\infty$ be the spectral  resolution of
$\mathcal{A}$. Moreover, $\{\varphi_i\}_{i=1}^\infty$ consisting of the orthnormal
basis w.r.t. inner product $(,)$ for $\mathfrak{H}$ is assumed. Define the
commutator $[\mathcal{A},\mathcal{B}]$ and the norm $\|\varphi\|$ by, respectively
$$
[\mathcal{A},\mathcal{B}]=\mathcal{A}\mathcal{B}-\mathcal{B}\mathcal{A}, \qquad \|\varphi\|^2=(\varphi,\varphi).
$$
Based on commutator algebra and  the Rayleigh-Ritz principle, M.S.
Ashbaugh and L. Hermi\cite{AshbHer1} obtained\\
\begin{thm}\label{AH}
 The eigenvalues $\ld_i$ of
the operator $\mathcal{A}$ satisfy the Yang-type inequality
\begin{equation}\label{AH-Ineq}
\sum_{i=1}^k(\ld_{k+1}-\ld_i)^2\rho_i\leq
\sum_{i=1}^k(\ld_{k+1}-\ld_i)\Lambda_i
\end{equation}

where $\rho_i, \Lambda$ are defined by, respectively,
$$
\begin{aligned}
\rho_i&=\sum_{k=1}^N \scal{[\mathcal{A},\mathcal{B}_k]\varphi_i}{\mathcal{B}_k\varphi_i}\\
\Lambda_i&=\sum_{k=1}^N\|[\mathcal{A},\mathcal{B}_k]\varphi_i\|^2.
\end{aligned}
$$
\end{thm}

 Applying  Theorem \ref{AH} to the weighted Hodge Laplacian $\dpf$, we have
\begin{lem} \label{Mainlemma} Let $(M^m,g)$ be an $m$-dimensional Riemannian manifold with Riemannian measure  $\dvolf$ and $u$ be a smooth function defined on $M^m$. For the eigenvalues $\Big\{\ldp_i\}_{i=1}^\infty$ of  the weighted Hodge Laplacian $\dpf$ (\ref{CM-operator}) acting on $p$-forms, we have,  $p \in \left\{0, 1, \dots,m\right\}$,
\begin{equation}\label{U-ineq}
\ba &\sum_{i=1}^k(\ldp_{k+1}-\ldp_i)^2\intm \abs{\n u}^2\abs{\varphi_i}^2\dvolf\\
&\leq \sum_{i=1}^k\left(\ldp_{k+1}-\ldp_i\right)\intm\Big((\Delta_{0,f} u)^2|\varphi_i|^2+4|\n_{\n u}\varphi_i|^2\\
&-4\langle\Delta_{0,f} u\varphi_i,\n_{\n u}\varphi_i\rangle\Big)\dvolf
\ea\end{equation}
where  $\{\varphi_i\}_{i=1}^\infty$ is a corresponding orthonormal basis of $p$-eigenforms, i.e.
\begin{equation*}
 \intm \scal{\varphi_i}{\varphi_j}\dvolf=\delta_{ij}.
\end{equation*}
\end{lem}

\begin{proof}
  It is easy to check that $\mathcal{A}=\dpf$ and $\mathcal{B}=u\in C^\infty(M^m,\R)$ satisfy
the conditions in Theorem \ref{AH}. Therefore, by the estimate of (\ref{AH-Ineq}), we have
\begin{equation}\label{bl}
\begin{aligned}
\sum_{i=1}^k\left(\ldp_{k+1}-\ldp_i\right)^2&
\int_M\langle[\dpf,u]\varphi_i,u\varphi_i\rangle \dvolf\\
& \leq\sum_{i=1}^k\left(\ldp_{k+1}-\ldp_i\right)\|[\dpf,u]\varphi_i\|^2,
\end{aligned}
\end{equation}
where
\begin{equation*}
\|[\dpf,u]\varphi_i\|^2= \intm \scal{[\dpf,u]\varphi_i}{[\dpf,u]\varphi_i} e^{-f}\mbox{dvol}.
\end{equation*}
By direct calculations, we have
\begin{align}
  [\dpf,u]\varphi_i &=[\Delta_p+\mathcal{L}_{\n f},u]\varphi_i \nonumber \\
   &=[\Delta_p,u]\varphi_i+[\mathcal{L}_{\n f},u]\varphi_i.\label{bl1}
\end{align}
From (\ref{L1}), we obtain
\begin{equation}\label{L4}
  [\mathcal{L}_{\n f},u]\varphi_i=g(\n f,\n u)\varphi_i.
\end{equation}
By (\ref{WB}), we have
\begin{align}
 [\Delta_p,u]\varphi_i&=[\n^*\n,u]\varphi_i\nonumber\\
 &=\Delta u\varphi_i-2\n_{\n u}\varphi_i\label{bl2}.
\end{align}
Therefore, from (\ref{bl1}) to (\ref{bl2}) we get
\begin{equation}\label{bl3}
  [\dpf,u]\varphi_i=\Delta_{0,f} u\varphi_i-2\n_{\n u}\varphi_i
\end{equation}
From (\ref{bl3}), we have
\begin{equation*}
\ba &\intm\langle[\dpf,u]\varphi_i,u\varphi_i\rangle e^{-f}\mbox{dvol}\\
   =&\intm\left\langle \Delta_{0,f} u \varphi_i-2\n_{\n u}\varphi_i,u\varphi_i\right\rangle\dvolf.
\ea
\end{equation*}
By integration by parts, we have
\begin{align*}
 2\intm\langle\n_{\n u}\varphi_i,u\varphi_i\rangle\dvolf &
  =\frac12\intm \scal{\n|\varphi_i|^2}{\n u^2}\dvolf \\
   &=\intm (u\Delta_{0,f} u-|\n u|^2)|\varphi_i|^2\dvolf .
\end{align*}
Finally, we obtain
\begin{equation}\label{bl4}
 \intm\langle[\dpf,u]\varphi_i,u\varphi_i\rangle e^{-f}\mbox{dvol}=\intm|\n u|^2|\varphi_i|^2\dvolf.
\end{equation}
On the other hand, using (\ref{bl3}), we get
\begin{equation}\label{bl5}
\ba\|[\dpf,u]\varphi_i\|^2=&\intm\Big((\Delta_{0,f} u)^2|\varphi_i|^2+4|\n_{\n u}\varphi_i|^2\\
&-4\langle\Delta_{0,f} u\varphi_i,\n_{\n u}\varphi_i\rangle\Big)\dvolf.
\ea\end{equation}
Inserting (\ref{bl4}) and (\ref{bl5}) into  (\ref{bl}), we obtain (\ref{U-ineq}).
\end{proof}

\begin{proof}[Proof of Theorem \ref{MainThm1}]
Letting $f=\frac12|x|^2$ and therefore $\dpx=\dpf$, substituting $u=x^A, A=1,\cdots,n$, the $p^{th}$ component of the isometric immersion $x=(x^1,\cdots, x^n): M^m\longrightarrow \R^n$ in (\ref{U-ineq}), and taking summation on $p$ from $1$ to $n$, we have
\begin{equation}\label{U-ineq1}
\ba &\sum_{A=1}^n\sum_{i=1}^k(\ldp_{k+1}-\ldp_i)^2\intm \abs{\n x^A}^2\abs{\varphi_i}^2\dvolx\\
&\leq \sum_{i=1}^k\left(\ldp_{k+1}-\ldp_i\right)\intm\sum_{A=1}^n\Big((\mathfrak{L} x^A)^2|\varphi_i|^2+4|\n_{\n x^A}\varphi_i|^2\\
&-4\langle\mathfrak{L} x^A\varphi_i,\n_{\n x^A}\varphi_i\rangle\Big)\dvolx,
\ea\end{equation}
where $\ml$ is the weighted Hodge Laplaican acting on functions given by (\ref{CM-operator1}).
From (\ref{xnorm}), we obtain
\begin{equation}\label{bl6}
\ba
\intm\sum_{p=1}^n|\n x^A|^2|\varphi_i|^2\dvolx
 &=m\intm|\varphi_i|^2\dvolx\\
 &=m.
\ea
\end{equation}
From (\ref{Lap-position}) and (\ref{bl5}), we get
\begin{equation}\label{bl7}
 \ba &\sum_{A=1}^n\|[\dpx,x^A]\varphi_i\|^2\\
 =& \sum_{A=1}^n\intm\bigg((\mathfrak{L} x^A)^2|\varphi_i|^2
 +4|\n_{\n x^A}\varphi_i|^2-4\langle
\mathfrak{L} x^A\varphi_i,\n_{\n x^A}\varphi_i\rangle\bigg)\dvolx\\
 =&\sum_{A=1}^n\intm\bigg((x^A)^2|\varphi_i|^2
+4|\n_{\n x^A}\varphi_i|^2-4\langle
 x^A\varphi_i,\n_{\n x^A}\varphi_i\rangle\bigg)\dvolx\\
\ea\end{equation}
Since $M^m$ is a compact self-shrinker,   by integration by parts and (\ref{ssf1}), we have
\begin{equation*}
 4\sum_{A=1}^n\intm\langle x^A\varphi_i,\n_{\n x^A}\varphi_i\rangle\dvolx=
 -\intm 2(m-|x|^2)|\varphi_i|^2\dvolx
\end{equation*}
Since $\displaystyle\sum_{A=1}^n|\n_{\n x^A}\varphi_i|^2=|\n \varphi|^2$, we have
\begin{equation}\label{bl8}
\ba \sum_{A=1}^n\|[\dpx,x^A]\varphi_i\|^2=&2m-\intm|x|^2|\varphi_i|^2\dvolx\\
&+4\intm|\n \varphi_i|^2\dvolx.
\ea\end{equation}
By integration by parts, from (\ref{WB}), (\ref{n5}) and (\ref{Cur-est}) , we have
\begin{equation}\label{bl9}
\begin{aligned}
  \intm|\n \varphi_i|^2\dvolx
  =&\intm\langle\n'^*\n \varphi_i,\varphi_i\rangle\dvolx\\
  =&\intm\langle(\dpx-\mathfrak{Ric}+\n(\n \frac{|x|^2}{2})) \varphi_i,\varphi_i\rangle\dvolx\\
  =&\ldp_i-\intm\langle\mathfrak{Ric}\varphi_i,\varphi_i\rangle\dvolx\\
  &+ \intm\langle\n(\n\frac{|x|^2}{2})\varphi_i,\varphi_i\rangle\dvolx,\\
\end{aligned}
\end{equation}
where $\n'^*$ is the adjoint operator of $\n$ with respect to the Riemannian measure $\dvolx$.
Therefore, we obtain
\begin{equation}\label{bl11}
\ba
\sum_{A=1}^n\|[\dpx,x^A]\varphi_i\|^2=&4\ldp_i+2m\\ &+4\intm\langle\n(\n\frac{|x|^2}{2})\varphi_i,\varphi_i\rangle\dvolx\\
&-4\intm\langle\mathfrak{Ric}\varphi_i,\varphi_i\rangle\dvolx\\
&-\intm|x|^2|\varphi_i|^2\dvolx.
\ea\end{equation}

From (\ref{bl}) and (\ref{bl11}), we get
\begin{equation*}\label{bl10}
\begin{aligned}
m\sum_{i=1}^k(\ldp_{k+1}-\ldp_i)^2
\leq & \sum_{i=1}^k\left(\ldp_{k+1}-\ldp_i\right)\left(4\ldp_i+2m\right.\\ &+4\intm\langle\n(\n\frac{|x|^2}{2})\varphi_i,\varphi_i\rangle\dvolx\\
&-4\intm\langle\mathfrak{Ric}\varphi_i,\varphi_i\rangle\dvolx\\
&\left.-\intm|x|^2|\varphi_i|^2\dvolx\right),
\end{aligned}
\end{equation*}
which completes the proof of Theorem \ref{MainThm1}.
\end{proof}

\begin{proof}[Proof of Corollary \ref{SH-inq1}]
From (\ref{U-ineq1}), (\ref{n5}) and (\ref{Cur-est}), we have
\begin{equation*}
 \ba &m\sum_{i=1}^k(\ldp_{k+1}-\ldp_i)^2\\
\leq &\sum_{i=1}^k\left(\ldp_{k+1}-\ldp_i\right)\left(4\ldp_i+2m+4\right.\\
&+4p\intm|H||h|\varphi_i|^2\dvolx\\
&\left.-\intm|x|^2|\varphi_i|^2\dvolx-4\intm\langle\mathfrak{Ric}\varphi_i,\varphi_i\rangle\dvolx\right)\\
\leq& 4\sum_{i=1}^k\left(\ldp_{k+1}-\ldp_i\right)\left[\ldp_i+\frac m2+1\right.\\
&\left.+\intm\left(p|H||h|-\Phi(H,h)-\frac14|x|^2\right)|\varphi_i|^2\dvolx\right]\\
\leq &4\sum_{i=1}^k\left(\ldp_{k+1}-\ldp_i\right)\left[\ldp_i+\frac m2+1\right.\\
&\left.+\max_{M^m}\left(p|H||h|-\Phi(H,h)-\frac14|x|^2\right)\right].
\ea\end{equation*}
\end{proof}

\section{Generalization of the Levitin-Parnovski inequality}
In this section, we will give the proof of Theorem \ref{MainThm2} by similar argument in \cite{IM12}. Firstly, we recall  the following algebraic identity obtained by Levitin and Parnovski (see identity 2.2 of Theorem 2.2 in \cite{LP02}).

\begin{lem}\label{LPT}
Let $\mathcal{L}$ and $\mathcal{G}$ be two self-adjoint operators with domains $D_{\mathcal{L}}$ and $D_{\mathcal{G}}$ contained in a same Hilbert space and such that
$G(D_{\mathcal{L}})\subseteq D_{\mathcal{L}}\subseteq D_{\mathcal{G}}$. Let $\lambda_j$ and $u_j$
be the eigenvalues and orthonormal eigenvectors of $\mathcal{L}$. Then, for each $j$,
\begin{equation}\label{LPT1}
\sum_{k=1}^\infty
\frac{|\langle[\mathcal{L},\mathcal{G}]u_j,u_k\rangle|^2_{L^2}}{\lambda_k-\lambda_j}
=\displaystyle{-\frac{1}{2}\langle[[\mathcal{L},\mathcal{G}],\mathcal{G}]u_j,u_j\rangle}_{L^2}
\end{equation}
(The summation is over all $k$ and is correctly defined even when
$\lambda_{k}=\lambda_{j}$ because in this case $\langle[\mathcal{L},\mathcal{G}]u_j,u_k\rangle=0$).
\end{lem}

\begin{proof}[Proof of Theorem \ref{MainThm2}]
By applying Lemma \ref{LPT} with
$\mathcal{L}=\dpx$ and $\mathcal{G}=x^A$, where $x^A$ is one of the components of the isometric immersion $x$, we have
\begin{equation}\label{LPT2}
\sum_{k=1}^\infty
\frac{|\langle[\dpx,x^A]u_j,u_k\rangle|^2_{L^2}}{\lambda_k-\lambda_j}
=\displaystyle{-\frac{1}{2}\langle[[\dpx,x^A],x^A]u_j,u_j\rangle}_{L^2}.
\end{equation}
From (\ref{bl3}), we have
\begin{align*}
[[\dpx,x^A],x^A]\varphi_i =&  [\dpx,x^A](x^A\varphi_i)-x^A([\dpx,x^A]\varphi_i) \\
=&\mathfrak{L}(x^A)x^A\varphi_i-2\n_{\n x^A}(x^A\varphi_i)-x^A(\mathfrak{L} x^A\varphi_i-2\n_{\n x^A}\varphi_i)\\
=&-2\n_{\n x^A}(x^A\varphi_i)+2x^A\n_{\n x^A}\varphi_i\\
=&-2|\n x^A|^2\varphi_i,
\end{align*}
hence
\begin{equation*}
-\frac{1}{2}\intm\langle[[\dpx,x^A],x^A]\varphi_i,\varphi_i\rangle\dvolx=\intm|\nabla
x^A|^2|\varphi_j|^2\dvolx.
\end{equation*}
From (\ref{LPT2}), we have
\begin{equation}\label{LPT3}
\ba&\intm|\nabla
x^A|^2|\varphi_i|^2\dvolx\\
=&\sum_{k=1}^\infty
\frac{1}{\ldp_k-\ldp_i}
\left(\intm\langle[\dpx,x^A]\varphi_i,\varphi_k\rangle\dvolx\right)^2.
\ea\end{equation}
For a fixed $i$, from the Gram-Schmidt orthogonalization, we can find the
coordinate system $\{x^A\}_{A=1}^n$ in Euclidean space $\Bbb R^n$ such that the matrix
$$
\left(\intm\langle[\dpx, x^A]\varphi_i,\varphi_{i+k}\rangle\dvolx\right)_{1\leq k,\;A \leq n}
$$
is a real upper triangular matrix. That is,
\begin{equation}\label{GS1}
\intm\langle[\dpx, x^A]\varphi_i,\varphi_{i+k}\rangle\dvolx=0, \qquad 1\leq k< A\leq n.
\end{equation}
By (\ref{GS1}), we can estimate the right hand side of (\ref{LPT3}) in the following
\begin{align}\label{LPT4}
&\sum_{k=1}^\infty
\frac{1}{\ldp_k-\ldp_i}
\left(\intm\langle[\dpx,x^A]\varphi_i,\varphi_k\rangle\dvolx\right)^2\nonumber  \\
=& \sum_{k=1}^{i-1}
\frac{1}{\ldp_k-\ldp_i}
\left(\intm\langle[\dpx,x^A]\varphi_i,\varphi_k\rangle\dvolx\right)^2\nonumber\\
&+\sum_{k=i+A}^{\infty}
\frac{1}{\ldp_k-\ldp_i}
\left(\intm\langle[\dpx,x^A]\varphi_i,\varphi_k\rangle\dvolx\right)^2\nonumber\\
\leq& \sum_{k=i+A}^{\infty}
\frac{1}{\ldp_k-\ldp_i}
\left(\intm\langle[\dpx,x^A]\varphi_i,\varphi_k\rangle\dvolx\right)^2\nonumber\\
\leq &
\frac{1}{\ldp_{i+A}-\ldp_i}
\sum_{k=1}^{\infty}\left(\intm\langle[\dpx,x^A]\varphi_i,\varphi_k\rangle\dvolx\right)^2\nonumber\\
=&\frac{1}{\ldp_{i+A}-\ldp_i}
\intm |[\dpx,x^A]\varphi_i|^2\dvolx\nonumber
\end{align}
where Parceval's identity is used in  the last equality.

Taking summation on $A$ from $1$ to $n$, from (\ref{LPT3}), (\ref{bl8}), (\ref{bl9}) and (\ref{WB}), we have
\begin{equation}\label{LPT5}
\begin{aligned}
 &\sum_{A=1}^n(\ldp_{i+A}-\ldp_i) \intm|\nabla
x^A|^2|\varphi_i|^2\dvolx\\
\leq&  \sum_{A=1}^n\intm |[\dpx,x^A]\varphi_i|^2\dvolx \\
=&2m-\intm|x|^2|\varphi_i|^2\dvolx+4\intm|\n \varphi_i|^2\dvolx\\
=&4\ldp_i+2m-\intm|x|^2|\varphi_i|^2\dvolx\\
  &-4\intm\langle\mathfrak{Ric}\varphi_i,\varphi_i\rangle\dvolx\\
  &+4\intm\langle\n(\n\frac{|x|^2}{2})\varphi_i,\varphi_i\rangle\dvolx.
  \end{aligned}
\end{equation}
Since $M^n$ is isometrically immersed in $\R^n$, it is easy to check
\begin{equation}\label{}
  \sum_{A=1}^n\ldp_{i+A}|\nabla x^A|^2\geq \sum_{l=1}^m\ldp_{i+l}.
\end{equation}
Therefore, we have
\begin{equation*}
\begin{aligned}
 \sum_{l=1}^m(\ldp_{i+l}-\ldp_i)\leq &4\ldp_i+2m-\intm|x|^2|\varphi_i|^2\dvolx\\
  &-4\intm\langle\mathfrak{Ric}\varphi_i,\varphi_i\rangle\dvolx\\
  &+4\intm\langle\n(\n\frac{|x|^2}{2})\varphi_i,\varphi_i\rangle\dvolx.
\end{aligned}
\end{equation*}
\end{proof}
\begin{proof}[Proof of Corollary \ref{Lower-inq}]
The proof of Corollary \ref{Lower-inq}  follows directly from (\ref{n5}) and (\ref{Cur-est}).
\end{proof}

\end{document}